%% file: coEulerian.tex
\documentclass[11pt,reqno]{amsart}
\usepackage{graphicx,tikz,amsmath,amsthm,verbatim,amssymb,lineno,eucal,titletoc,complexity}
\usepackage{algorithm}
\usepackage{algpseudocode}
\usepackage{bbm}
\usepackage{hyperref}
\hypersetup{colorlinks}
\newcommand{\arxiv}[1]{{\tt \href{http://arxiv.org/abs/#1}{arXiv:#1}}}

\newcommand{\old}[1]{}

\newtheorem{thm}{Theorem}[section]
\newtheorem{prop}[thm]{Proposition}

\newtheorem{lem}[thm]{Lemma}
\newtheorem{cor}[thm]{Corollary}

\theoremstyle{remark}

\numberwithin{counter}{section}

\theoremstyle{definition}
\newtheorem{defn}[thm]{Definition}
 \newtheorem*{problem*}{\protect\problemname}

\def\mod{\mbox{ mod }}

\def\pp{\mathbf{p}}

\def\xx{\mathbf{x}}
\def\yy{\mathbf{y}}

\def\00{\mathbf{0}}

\def\ord{\mathrm{ord}}

\def\Rec{\mathop{\mathrm{Rec}}}

\def\zero{\mathbf{0}}
\def\one{\mathbf{1}}

\def\N{\mathbb{N}}
\def\Z{\mathbb{Z}}

\newlength{\problemoffset}
\newcommand{\decision}[3]{%     \decision{NAME}{INSTANCE}{QUESTION}
\medskip
{\textsc{#1}:}
\smallskip
\begin{list}{}{
\setlength{\leftmargin}{\problemoffset}
\setlength{\rightmargin}{\problemoffset}
\setlength{\parsep}{1pt}
\setlength{\itemsep}{2pt}
\setlength{\topsep}{\itemsep}
\setlength{\partopsep}{\itemsep}
}
\item
{\textbf{Given} #2}
\item
{\textbf{Decide} #3}
\end{list}
\medskip
}

\makeatletter
\def\BState{\State\hskip-\ALG@thistlm} % for algorithm
\makeatother

\begin{document}

\title{CoEulerian graphs}

\author{Matthew Farrell \and Lionel Levine}

\date{September 4, 2015}
\thanks{This research was supported by NSF grants \href{http://www.nsf.gov/awardsearch/showAward?AWD_ID=1243606}{DMS-1243606} and \href{http://www.nsf.gov/awardsearch/showAward?AWD_ID=1455272}{DMS-1455272} and a Sloan Fellowship.}
\address{Matthew Farrell, Department of Mathematics, Cornell University, Ithaca, NY 14850. msf235@cornell.edu}
\address{Lionel Levine, Department of Mathematics, Cornell University, Ithaca, NY 14850. \url{http://www.math.cornell.edu/~levine/}}

\maketitle

\begin{abstract} 
We suggest a measure of ``Eulerianness'' of a finite directed graph and define a class of ``coEulerian'' graphs.  These are the graphs whose Laplacian lattice is as large as possible.  As an application, we address a question in chip-firing posed by  Bj\"{o}rner, Lov\'{a}sz, and Shor in 1991, who asked for ``\emph{a characterization of those digraphs and initial chip configurations that guarantee finite termination.}''  Bj\"{o}rner and Lov\'{a}sz gave an exponential time algorithm in 1992.  We show that this can be improved to linear time if the graph is coEulerian, and that the problem is $\NP$-complete for general directed multigraphs. 
\end{abstract}

\section{Introduction}

In this paper $G=(V,E)$ will always denote a finite directed graph, with loops and multiple edges permitted.  We assume throughout that $G$ is strongly connected: for each $v,w \in V$ there are directed paths from $v$ to $w$ and from $w$ to $v$.
Trung Van Pham \cite{Trung} introduced the quantity
	\[ M_G = \gcd \{ \kappa(v) | v \in V\} \]
where $\kappa(v)$ is the number of spanning trees of $G$ oriented toward $v$.  We will see that $M_G$, which we will call the \textbf{Pham index} of the graph $G$, can be interpreted as a measure of ``Eulerianness''.   

A finite directed multigraph $G$ is called \textbf{Eulerian} if it has an Eulerlian tour (a closed path that traverses each directed edge exactly once). We are going to take the view that Eulerianness is an algebraic property of the \textbf{graph Laplacian} $\Delta$ acting on integer-valued functions $f \in \Z^V$ by
	\begin{equation} \label{e.thelaplacian} \Delta f (v) = d_v f(v) - \sum_{\mathrm{head}(e)=v} f(\mathrm{tail}(e)). \end{equation}
Here $d_v$ is the outdegree of vertex $v$.
The context is the following well-known equivalence, where $\one$ denotes the constant function $\one(v)=1$ for all $v \in V$.

\begin{samepage}
\begin{prop} \label{p.Eul}
The following are equivalent for a strongly connected directed multigraph $G=(V,E)$.
\begin{enumerate}
\item $\ker (\Delta : \Z^V \to \Z^V) = \Z \one$.
\item $M_G = \kappa(s)$ for all $s \in V$.
\item $G$ is Eulerian.
\end{enumerate}
\end{prop}
\end{samepage}

Our main result is in some sense dual to Proposition~\ref{p.Eul}: it gives several equivalent characterizations of the graphs with Pham index $1$. These \textbf{coEulerian} graphs are the farthest from being Eulerian.

\old{
A function $f$ is called harmonic if $\Delta f = 0$.  Compared to other strongly connected graphs, the Eulerian graphs have ``a lot'' of integer-valued harmonic functions. This might seem a surprising thing to say, since the harmonic functions are just the constants!  But for $G$ finite and strongly connected, the kernel of $\Delta$ is always spanned by a single integer function $\pi$, which moreover has all values strictly positive.  If $G$ is not Eulerian then $\pi \neq \one$, which means that $G$ has fewer harmonic functions (with values bounded by a fixed constant $C$, say) than does an Eulerian graph. 
}

Our motivation for considering coEulerian graphs and the Pham index comes from chip-firing, which we now describe. 
A \textbf{chip configuration} on $G$, or simply \textbf{configuration} for short, is a function $\sigma : V \to \Z$. If $\sigma(v)>0$ we think of a pile of $\sigma(v)$ chips at vertex $v$, and if $\sigma(v)<0$ we think of a hole waiting to be filled by chips.
Denoting by $d_{v}$ the outdegree of vertex $v$, we say that
$v$ is \textbf{stable} for $\sigma$ if $\sigma(v)<d_{v}$,
and \textbf{active} for $\sigma$ otherwise. A vertex $v$ can \textbf{fire} by sending one chip along each outgoing edge, resulting in the new configuration
	\[ \sigma' = \sigma - \Delta \delta_v \]
where $\Delta$ is the graph Laplacian \eqref{e.thelaplacian} and $\delta_v(w)$ is $1$ if $v=w$ and $0$ otherwise.  Concretely, we may think of $\sigma, \sigma'$ as column vectors and $\Delta \delta_v$ as a column of the matrix
\[
\Delta_{vw}=\begin{cases}
-d_{wv}, & v\neq w\\
d_{v}-d_{vv}, & v=w
\end{cases}
\]
where $d_{wv}$ denotes the number of directed edges of $G$ from $w$ to $v$. More generally, we can specify a \textbf{firing vector} $\textbf{x} \in \N^V$ and fire each vertex $v$ a total of $\textbf{x}(v)$ times, resulting in the configuration $\sigma'=\sigma-\Delta\textbf{x}$. Here and throughout, $\N = \{0,1,2,\ldots\}$.

A \textbf{legal firing sequence} is a finite sequence of configurations $\sigma_0, \ldots, \sigma_k$ such that each $\sigma_i$ for $i=1,\ldots,k$ is obtained from $\sigma_{i-1}$ by firing a vertex that is active for $\sigma_{i-1}$. A configuration $\sigma$ is called \textbf{stable} if $\sigma(v)<d_v$ for all $v\in V$.   
We say that $\sigma$ \textbf{stabilizes} if there is legal firing sequence $\sigma = \sigma_0, \ldots, \sigma_k$ such that $\sigma_k$ is stable. Bj\"{o}rner, Lov\'{a}sz, and Shor posed the following problem in 1991 \cite{BLS}. \hypertarget{d.halting}~

\decision{The halting problem for chip-firing}
{the adjacency matrix of a finite, strongly connected multigraph $G$ and a chip configuration $\sigma$ on $G$ with $\sigma\geq0$,}
{whether $\sigma$ stabilizes.}

Write $|\sigma| = \sum_{v \in V} \sigma(v)$ for the total number of chips. This quantity is conserved by firing (since $|\Delta \delta_v | = 0$ for all $v \in V$). The maximal stable configuration 
	\[ \sigma_{\max}(v)=d_{v}-1 \]
has $|\sigma_{\max}| = \# E - \# V$.  By the pigeonhole principle, any configuration $\sigma$ with $|\sigma| > \# E - \# V$ has at least one unstable vertex, so such $\sigma$ does not stabilize. A natural question arises: Which directed graphs have the property that every chip configuration of $\#E - \# V$ chips stabilizes? These graphs are the subject of our main result.

We write $\Z^V_0$ for the set of $\sigma \in \Z^V$ such that $|\sigma|=0$.

\begin{samepage}
\begin{thm}
\label{t.main} The following are equivalent for a strongly connected directed multigraph $G=(V,E)$.
\begin{enumerate}
\item $\mathrm{Im} (\Delta : \Z^V \to \Z^V) = \Z^V_0$.
\item $M_G=1$.
\item A chip configuration $\sigma$ on $G$ stabilizes if and only if $\left|\sigma\right| \leq \#E - \# V$.
\item For all $s \in V$, the sandpile group $K(G,s)$ is cyclic with generator $\overline{\beta_s}$.
\item For some $s \in V$, the sandpile group $K(G,s)$ is cyclic with generator $\overline{\beta_s}$.
%% note: I switched the order of 4 and 5 to match Prop~\ref{p.Eul2}.
\end{enumerate}
\end{thm}
\end{samepage}

Items (1) and (2) are in some sense dual to their counterparts in Proposition~\ref{p.Eul}, so we propose the term \textbf{coEulerian} for a graph satisfying the equivalent conditions of Theorem~\ref{t.main}. The sandpile group $K(G,s)$ and $\overline{\beta_s}$ are defined below in Section~\ref{s.mainsection}. For a dual counterpart to item (4), see Proposition~\ref{p.Eul2}(4).

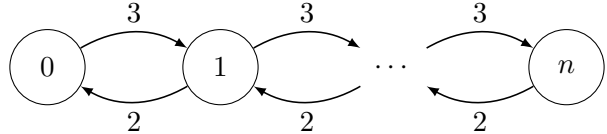
\begin{figure}[here]
\centering
\input{nodes2.tex}
\caption{Example of a coEulerian graph: a path of length $n$ with edge multiplicities $2$ to the right and $3$ to the left.  It has $\kappa(v) = \pi(v) = 2^v 3^{n-v}$ spanning trees oriented toward $v$, so its Pham index is $M_G = \gcd(2^n, 2^{n-1}3, \ldots, 3^n) = 1$.}
\label{f.path}
\end{figure}

\subsection{History}

Let us call a graph \emph{bidirected} if it is obtained from an undirected graph by replacing each undirected edge $\{u,v\}$ by a pair of directed edges $(u,v)$ and $(v,u)$. All bidirected graphs are Eulerian. 

One of the earliest results in chip-firing is the following observation of Tardos. 

\begin{lem} \cite[Lemma 4]{Tardos} \label{l.tardos}
Let $\sigma$ be a configuration on a bidirected graph $G$.
If there is a legal firing sequence for $\sigma$ in which every vertex of $G$ fires at least once, then $\sigma$ does not stabilize.
\end{lem}
% This lemma is true more generally if G is Eulerian.

Tardos used Lemma~\ref{l.tardos} to prove that for any configuration $\sigma$ on a simple bidirected $n$-vertex graph, if $\sigma$ stabilizes then it does so in $O(n^4)$ firings.  Eriksson showed, however, that on a general directed graph a configuration may require an exponential number of firings to stabilize \cite{Eriksson}.  Bj\"{o}rner and Lov\'{a}sz \cite{BL92} generalized the ``at least once'' condition of Lemma~\ref{l.tardos} to directed graphs as follows.

\begin{lem} \cite{BL92} \label{l.bl}
For every strongly connected multigraph $G$ there is a unique primitive $\pi \in \N^V$ such that $\Delta \pi = \zero$. If there is a legal firing sequence for $\sigma$ in which every vertex $v$ fires at least $\pi(v)$ times, then $\sigma$ does not stabilize.
\end{lem}

This gives a procedure for deciding the \hyperlink{d.halting}{\textsc{halting problem for chip-firing}}: perform legal firings in any order until either you reach a stable configuration or the criterion of Lemma~\ref{l.bl} certifies that $\sigma$ will not stabilize.  There is only one problem: the values $\pi(v)$ may be exponentially large.  Figure~\ref{f.path} shows a coEulerian graph on vertex set $\{0,1,\ldots,n\}$ with $\pi(v)=2^{v}3^{n-v}$.  The algorithm just described would run for exponential time on this graph, but Theorem~\ref{t.main} gives a much faster algorithm to decide the halting problem for chip-firing on any coEulerian graph: count the total number of chips and compare to $\#E - \#V$. As far as we are aware, this is the first progress on the halting problem for chip-firing on directed graphs since the work of Bj\"{o}rner and Lov\'{a}sz \cite{BL92}.

\subsection{Related work}

Pham \cite{Trung} introduced the index $M_G$ to answer a question posed in \cite{HLMPPW}: Which directed graphs $G$ have the property that all unicycles of $G$ lie in the same orbit of the rotor-router operation?  He showed that $G$ has this property if and only if $M_G=1$, and that in general the number of orbits is $M_G$.
%In other words, the graphs with this property are precisely the coEulerian graphs.

The \hyperlink{d.halting}{\textsc{halting problem for chip-firing}} is a special case of the halting problem for a class of automata networks called \emph{abelian networks}. A polynomial time algorithm to decide if a given abelian network halts on \emph{all} inputs appears in \cite{BL2}, where it is remarked that the problem of deciding whether a given abelian network halts on a \emph{given} input is a subtler problem.  The \textsc{halting problem for chip-firing} is of this latter type (the ``input'' to the abelian network is the chip configuration $\sigma$).

\subsection{Outline}

The next section is devoted to the proofs of Proposition~\ref{p.Eul} and Theorem~\ref{t.main}, and concludes with Proposition~\ref{p.cactus} characterizing the graphs that are both Eulerian and coEulerian. In Section~\ref{s.NPcomplete} we show that despite its being easy for Eulerian graphs and coEulerian multigraphs, the \textsc{halting problem for chip-firing} on finite directed multigraphs is $\NP$-complete in general. One ingredient in the proof is Theorem~\ref{t.latticelaplacian}, which expresses an arbitrary $(n-1)$-dimensional lattice in $\Z_0^n$ as the Laplacian lattice of a strongly connected multigraph.

\section{Sandpiles and the halting problem}
\label{s.mainsection}

To prove Theorem~\ref{t.main} we will compare chip-firing with and without a sink vertex.  This kind of comparison appears also in the study of the abelian sandpile threshold state \cite{L14}, and in the extension of the Biggs-Merino polynomial to Eulerian graphs \cite{PP} and to all strongly connected graphs \cite{Chan}.  Sections~\ref{s.background} and~\ref{s.cyclic} review the relevant background on chip-firing and the sandpile group. In Section~\ref{s.isom} we relate the sandpile groups with and without sink, and in Section~\ref{s.proofs} we prove the results stated in the introduction.

\subsection{Background}
\label{s.background}

The following result frees us from considering only legal firing
sequences in looking for an answer to the halting problem for chip-firing.

\begin{lem} 
\label{l.lap}
\emph{(Least Action Principle, \cite[Lemma 4.3]{BL1})} Let $\sigma$ be a chip configuration on a finite directed graph. Then $\sigma$ stabilizes if and only if there exists an $\mathbf{x}\in\mathbb{N}^{V}$ such that $\sigma-\Delta\mathbf{x}$ is stable.
\end{lem}

A sizable portion of the ground soon to be covered is motivated by
the following principle: in looking for a stabilizing firing sequence,
instead of firing willy-nilly we can establish some structure by
designating a special vertex $s$ as the \textbf{sink}, which fires only if no other vertex is active.
We fire active, nonsink vertices until all nonsink vertices are stable. At this point if the sink is stable we are done; otherwise, we fire
the sink (once) and repeat.

The \textbf{reduced Laplacian} $\Delta_{s}$ is the matrix obtained by deleting the row and column of $\Delta$
corresponding to the sink $s$. To emphasize the distinction between $\Delta$ and $\Delta_s$, we will sometimes refer to $\Delta$ as the \textbf{total Laplacian}. 
In what follows we will sometimes identify the vertex set $V$ with $\{1,\ldots,n\}$ and set $s=n$. 

\begin{defn}
Let $G=(V,E)$ be a finite strongly connected multigraph and fix $s \in V$. The \textbf{sandpile group }of $G$ with sink $s$ is
the group quotient
\[
K(G,s)=\mathbb{Z}^{n-1}/\Delta_{s}\mathbb{Z}^{n-1}
\]
where $\Delta_{s}\mathbb{Z}^{n-1}$ is the integer column-span of
$\Delta_{s}$.
\end{defn}

A \textbf{sandpile} is a chip configuration $\eta\in\mathbb{Z}^{n-1}$ indexed by the nonsink vertices. When we wish to emphasize that a chip configuration is defined also at the sink, we call it a \textbf{total configuration}. One can imagine that a sandpile is composed of sand grains which behave just like chips except that they are small enough to disappear down the sink.
The definitions ``stable'' and ``firing
vector'' have obvious analogues for sandpiles: a sandpile
$\eta$ is stable if $\eta(i)<d_{i}$ for all $v_{i}\neq s$; and firing vectors for sandpiles live in $\mathbb{Z}^{n-1}$. The sandpile group treats two sandpiles as equivalent if one can be obtained from the other by firing nonsink vertices. We write $\overline{\eta}$ for the equivalence class of $\eta$ in $K(G,s)$. 

On a strongly connected graph, every sandpile stabilizes, and its stabilization does not depend on the order of firings \cite[Lemmas 2.2 and 2.4]{HLMPPW}; we denote the stabilization of $\eta$ by $\eta^{\circ}$. Next we recall the connection between sandpiles and spanning trees.

\begin{defn}
An \textbf{oriented spanning tree} of a directed graph $G=(V,E)$ rooted
at $s\in V$ is a spanning subgraph $T=(V,A)$ such that\end{defn}
\begin{enumerate}
\item Every vertex $v\neq s$ has outdegree $1$ in $T$.
\item $s$ has outdegree $0$ in $T$.
\item $T$ has no oriented cycles.
\end{enumerate}
Hence an oriented spanning tree has as its limbs edges that point
toward the root. Let $\kappa(s)$ denote the number of oriented
spanning trees in $G$ rooted at $s$.
\begin{thm}
 \emph{(Matrix tree theorem \cite[Theorem 5.6.8]{Stanley} and \cite[Lemma 2.8]{HLMPPW})} 
For a finite strongly connected multigraph $G$ and a vertex $s$, 
  \[ \kappa(s) = \det \Delta_s = \#K(G,s). \]
\end{thm}

Note that if $G$ is strongly connected then it has at least one spanning tree rooted at $s$, so $\Delta_{s}$ is invertible; since the rows of $\Delta$ sum to $\textbf{0}$, this implies that $\Delta$ has rank $n-1$. 

There is a natural representative for each equivalence
class of $K(G,s)$. To describe this representative, we say that a sandpile $\eta$
is \textbf{accessible} if from any other sandpile it is possible to
obtain $\eta$ by adding a nonnegative number of sand grains at each vertex and then selectively firing
active vertices. A sandpile that is both stable and accessible is
called \textbf{recurrent}.

\begin{thm} \cite[Cor.\ 2.16]{HLMPPW}
\label{t.isom}
The set $\Rec (G,s)$ of all recurrent sandpiles is an abelian
group under the operation 
%$(\eta,\xi)\mapsto(\eta+\xi)^{\circ}$,
	\[ \eta \oplus \xi := (\eta+\xi)^{\circ} \]
and it is isomorphic via the inclusion map to the sandpile group $K(G,s)$.
\end{thm}

The \textbf{recurrent identity element} $e_s \in \Rec(G,s)$ is the unique recurrent sandpile in $\Delta_s \Z^{n-1}$.  The recurrent representative $\eta_{\text{rec}}$ of a sandpile $\eta$ can be found by adding the identity and stabilizing:
	\[ \eta_{\text{rec}} = (\eta+e_s)^\circ. \]
Dhar's burning test \cite{Dha90} determines whether a given sandpile on an Eulerian graph is recurrent.  Speer~\cite{Spe93} generalized the burning test to directed graphs. Dhar's and Speer's tests are closely related to Lemmas~\ref{l.tardos} and~\ref{l.bl} respectively.

\subsection{Cyclic subgroups of the sandpile group}
\label{s.cyclic}

For $s,v \in V$ let $\beta_s(v) = d_{sv}$, the number of directed edges from $s$ to $v$.
In accordance with our principle of controlled sink firing, given a recurrent sandpile $\eta$ we are interested in 
\[
C_{\eta}=\{\left(\eta+k\beta_{s}\right)^{\circ}:k\in\mathbb{N}\},
\]
the set of sandpiles obtainable from $\eta$ by firing the sink $s$ some nonnegative number of times and then stabilizing. Note that starting with a recurrent sandpile, adding sand grains to the nonsink vertices and then stabilizing results in another recurrent sandpile; so all sandpiles in 
$C_{\eta}$ are recurrent. Note that 
\[
\left(\eta+\beta_{s}\right)^{\circ}=\left(\eta+e_{s}+\beta_{s}\right)^{\circ}=\eta\oplus\gamma_{s}
\]
where $\gamma_{s}=(e_{s}+\beta_{s})^{\circ}$ is the recurrent representative of $\beta_s$. It follows
that 
	\[ C_{\eta}=\eta\oplus\left\langle \gamma_{s}\right\rangle \]
where $\left\langle \gamma_{s}\right\rangle $ denotes the cyclic
subgroup of $\Rec (G,s)$ generated by $\gamma_{s}$. 

To investigate these cosets of $\langle \gamma_s \rangle$, we recall the period vector introduced by Bj\"{o}rner and Lov\'{a}sz.

\begin{defn} \cite{BL92}
Given a graph $G$ with total Laplacian $\Delta$, a vector $\mathbf{p}\in\mathbb{N}^{n}$
is called a \textbf{period vector} for $G$ if $\mathbf{p} \neq \zero$ and $\Delta\mathbf{p}=\zero$.
A period vector is \textbf{primitive} if the greatest common divisor of its entries is $1$.
\end{defn}
In other words, a period vector $\mathbf{p}$ has the property that firing each vertex
$v\in V$ a total of $\mathbf{p}(v)$ times results in no net movement of chips.
The following lemma sums up some useful properties of period vectors.

\begin{lem} \cite[Prop.\ 4.1]{BL92} 
\label{lem:period} 
A strongly connected multigraph $G$ has a unique primitive period vector $\pi_{G}$. All entries of $\pi_G$ are strictly positive, and all period vectors of $G$ are of the form $k\pi_{G}$ for $k=1,2,\ldots$. 
Moreover, if $G$ is Eulerian, then $\pi_{G}=\mathbf{1}$.
\end{lem}

A consequence of the strict positivity of $\pi_G$ that we will use several times is that $\Delta \Z^n = \Delta \N^n$.

We now introduce a very special period vector. Recall that $\kappa(v)$
denotes the number of spanning trees of $G$ oriented toward $v$.
\begin{lem}[\cite{Aldous,Broder}]
\label{lem:kerper} $\Delta \kappa = \zero$.
\end{lem}
Recall the \textbf{Pham index} $M=M_G$, defined as the greatest common divisor of the spanning tree counts $\left\{ \kappa(v)|v\in V\right\} $.
By Lemmas \ref{lem:period} and \ref{lem:kerper}, the vector $\pi=\frac{1}{M} \kappa$
is the unique primitive period vector of $G$.  

Next we argue that $\pi(s) = \ord(\gamma_s)$, the order of $\gamma_s$ in the group $\Rec(G,s)$.  Fixing a positive integer $m$, we have that $m \overline{\beta_s}$ is trivial in $K(G,s)$ if and only if there exists $\yy\in \Z^{n-1}$ such that $m \beta_s = \Delta_s \yy$. Setting $\xx = (\yy,0) \in \Z^n$ and noting that $\beta_s$ is the restriction of $-\Delta \delta_s$ to the nonsink vertices, such $\yy$ exists if and only if there is a vector $\xx \in \Z^{n}$ such that $\xx(s) = 0$ and
	$ \Delta (\xx + m\delta_s) = \zero $.
(The equality in the sink coordinate of the first equation follows from the equality in the nonsink coordinates because the sum of all the coordinates is $0$.) 
Setting $\pp = \xx + m\delta_s$, such $\xx$ exists if and only if there is a vector $\pp \in \Z^n$ such that
	\[ \Delta \pp = \zero \qquad \mbox{and} \qquad \pp(s)=m. \]
By Lemma~\ref{lem:kerper}, since the kernel of $\Delta$ is one-dimensional and $m$ is positive, such $\pp$ must be a positive integer multiple of the primitive period vector $\pi$.  
Such $\pp$ exists if and only if $\pi(s)$ divides $m$. Thus $\pi(s)$ is the order of $\overline{\beta_s}$ in $K(G,s)$, which by Theorem~\ref{t.isom} is the order of $\gamma_s$ in $\Rec(G,s)$.  Recalling that $\pi(s)=\kappa(s)/M$, we conclude the following.

\begin{lem} \label{lem:order} \cite[Lemma~6]{Trung}
For any choice of sink $s$, we have that 
\[
\ord (\gamma_{s})=\kappa(s)/M= \# \Rec(G,s) /M
\]
Thus, $M = \# \Rec (G,s)/\left\langle \gamma_{s}\right\rangle$ is the number of distinct cosets of $\left\langle \gamma_{s}\right\rangle $ in $\Rec(G,s)$.
\end{lem}

\subsection{Comparison of sandpile groups with and without sink}
\label{s.isom}

We now investigate the structure of the quotient group $\Rec (G,s)/\left\langle \gamma_{s}\right\rangle $.
Recall that $\beta_{s}$ is the sandpile $\beta_{s}(v)=d_{sv}$, where $s$ is the designated sink vertex and that $\overline{\beta_s}$ is the equivalence class of $\beta_s$ in $K(G,s)$. As before we write $\Z^n_0$ for the group of vectors
in $\mathbb{Z}^{n}$ with coordinates summing to $0$.
\begin{thm}
\label{thm:groupiso} For any strongly connected multigraph $G$ and any vertex $s$,
\[
\Rec (G,s)/\left\langle \gamma_{s}\right\rangle \cong K(G,s)/\left\langle \overline{\beta_{s}}\right\rangle \cong \Z^n_0/\Delta\mathbb{Z}^{n}.
\]

\end{thm}

The meat of the proof for this theorem is packaged in the following workhorse lemma. 
To translate between sandpiles and total configurations, we introduce some notation: Given a vector $\xx\in\mathbb{Z}^{n}$,
we denote by $\tilde{\xx}$ the restriction of $\xx$ to the nonsink vertices; and given $\eta\in\mathbb{Z}^{n-1}$, we write $\eta_{k}$
for the extension of $\eta$ to $\mathbb{Z}^{n}$ such that $\left|\eta_{k}\right|=k$.
%Note that $\eta$ and $k$ together uniquely determine $\eta_{k}$, since $\left|\eta\right|+\eta_{k}(s)=k$.

\begin{lem}
\label{cycleequiv}Let $\sigma,\tau\in\Z^n$ with $|\sigma|=|\tau|$. Then the following are equivalent.\end{lem}
\begin{enumerate}
\item $\sigma\equiv\tau\mod\Delta\mathbb{Z}^{n}$ 
\item $\tilde{\sigma}\equiv\tilde{\tau}\mod\Delta_{s}\mathbb{Z}^{n-1}+\mathbb{Z}\beta_{s}$
\item $\left(\tilde{\sigma}+e_{s}\right)^{\circ}\in\left(\tilde{\tau}+e_{s}\right)^{\circ}\oplus\left\langle \gamma_{s}\right\rangle $\end{enumerate}
\begin{proof}
($1\iff2$) Assume ($1$), and let $m=\left|\sigma\right|=\left|\tau\right|$. Recall
that $\sigma_{k}$ denotes the extension of $\sigma$ to $\mathbb{Z}^{n}$
such that $\left|\sigma_{k}\right|=k$. We observe that ($1$) holds if and only if there is an $\mathbf{x}\in\mathbb{Z}^{n}$ such that $\sigma=\tau-\Delta\mathbf{x}$.
If $\sigma=\tau-\Delta\mathbf{x}$, then 
\[
\sigma=\tau-\Delta\mathbf{x}=\tau-\Delta\left[\begin{array}{c}
\tilde{\mathbf{x}}\\
0
\end{array}\right]-\Delta\left[\begin{array}{c}
0\\
\mathbf{x}(s)
\end{array}\right]=\tau-\left[\begin{array}{c}
\Delta_{s}\tilde{\mathbf{x}}\\
a
\end{array}\right]-\mathbf{x}(s)\mathbf{c}_{s}
\]
where $\mathbf{c}_{s}$ denotes the column of $\Delta$ corresponding
with the sink and $a$ is the dot product of the $n$th row of $\Delta$
with $(\tilde{\mathbf{x}},0)$. Since $\beta_{s}(i)=-\mathbf{c}_{s}(i)$
for each $i\neq s$, it follows that $\tilde{\sigma}=\tilde{\tau}-\Delta_{s}\tilde{\mathbf{x}}+\mathbf{x}(s)\beta_{s}$.
Going the other way, we assume that $\tilde{\sigma}=\tilde{\tau}-k\beta_{s}-\Delta_{s}\tilde{\mathbf{x}}$
for some $k\in\mathbb{N}$ and $\tilde{\mathbf{x}}\in\mathbb{N}^{n-1}$.
Let $\sigma'$ be the total configuration 
\[
\sigma'=\tau-\Delta\left[\begin{array}{c}
\tilde{\mathbf{x}}\\
k
\end{array}\right].
\]
Then $\sigma'(i)=\mbox{\ensuremath{\tilde{\sigma}}}(i)$ for all $i\neq s$
and $\left|\sigma'\right|=\left|\tau\right|$. Since $\sigma(s)$
is determined by $|\sigma|$ and $|\sigma|=|\tau|$, we have that $\sigma'=\sigma$.

$(2\iff3)$ Note that $(3)$ is equivalent to the existence of an $\mathbf{x}$ such that \[ \left(\tilde{\sigma}+e_{s}\right)^{\circ}\equiv\left(\tilde{\tau}+e_{s}\right)^{\circ}+\mathbf{x}(s)\left(\beta_{s}+e_{s}\right)^{\circ}-\tilde{\Delta}_{s}\tilde{\mathbf{x}} \]
which in turn is equivalent to the congruence
$\tilde{\sigma}\equiv\tilde{\tau}\mod\tilde{\Delta}_{s}\mathbb{Z}^{n-1}+\mathbb{Z}\beta_{s}$.
\end{proof}

\begin{proof}[Proof of Theorem \ref{thm:groupiso}] Define a map $\phi:K(G,s)/\left\langle \overline{\beta_{s}}\right\rangle \to \Z^n_0/\Delta\mathbb{Z}^{n}$
sending 
\[ \overline{\eta}\mod\left\langle \overline{\beta_{s}}\right\rangle \quad \mapsto \quad \eta_{0}\mod\Delta\mathbb{Z}^{n}. \]
Let $\eta,\xi\in\Z^{n-1}$. If $\overline{\eta}\equiv\overline{\xi}\mod\left\langle \overline{\beta_{s}}\right\rangle $,
then by Lemma \ref{cycleequiv} we have that $\eta_0\equiv\xi_0\mod\Delta\mathbb{Z}^{n}$,
so that $\phi$ is well-defined. The equation $\eta_{0}+\xi_{0}=\left(\eta+\xi\right)_{0}$
is immediate from the definition, so that $\phi$ is a homomorphism.
The map $\phi$ is also surjective, since for each $\sigma\in \Z^n_0$
there is a corresponding $\tilde{\sigma}\in\mathbb{Z}^{n-1}$ , and
$\phi\left(\overline{\tilde{\sigma}}\mod\left\langle \overline{\beta_{s}}\right\rangle \right)=\sigma\mod\Delta\mathbb{Z}^{n}$.
We now show that $\phi$ is injective to complete the proof that $\phi$
is an isomorphism. Suppose that $\sigma\equiv\tau\mod\Delta\mathbb{Z}^{n}$.
Then by Lemma \ref{cycleequiv} we have that $\overline{\tilde{\sigma}}\equiv\overline{\tilde{\tau}}\mod\left\langle \overline{\beta_{s}}\right\rangle $
and the theorem is proved.
\end{proof}

\subsection{Eulerian and CoEulerian Graphs}
\label{s.proofs}

We now prove the two results stated in the introduction. We also supplement Proposition~\ref{p.Eul} with three equivalent conditions about the sandpile group.  

\begin{prop} 
\label{p.Eul2}
The following are equivalent for a strongly connected directed multigraph $G=(V,E)$.
\begin{enumerate}
\item $\ker (\Delta : \Z^V \to \Z^V) = \Z \one$.
\item $M_G = \kappa(s)$ for all $s \in V$.
\item $G$ is Eulerian.
\item For all $s \in V$, the element $\overline{\beta_s}$ is trivial in the sandpile group $K(G,s)$.
\item $K(G,s) \cong \Z_0^V/ \Delta \Z^V$ for all $s \in V$.
\item $K(G,s) \cong K(G,s')$ for all $s,s' \in V$.
%\item [NOT equivalent!] $K(G,s) \cong \Z_0^V/ \Delta \Z^V$ for some $s \in V$.
\end{enumerate}
\end{prop}

\begin{proof}
($1 \iff 3$) We have $\Delta \one = \zero$ if and only if the indegree of each vertex equals its outdegree. By \cite[Theorem 5.6.1]{Stanley}, for $G$ strongly connected this degree condition is equivalent to the existence of an Eulerian tour.

($1 \implies 4$) By definition, $\beta_s$ is the restriction of $-\Delta \delta_s$ to the nonsink coordinates. If $\Delta \one = \zero$ then $-\Delta \delta_s = \Delta (\sum_{v \neq s} \delta_v)$, so $\beta_s = \Delta_s (\sum_{v \neq s} \delta_v) \in \Delta_s \Z^{n-1}$.

($4 \implies 5$) This follows directly from Theorem \ref{thm:groupiso}.

($5 \implies 6$) Trivial.

($6 \implies 2$) If $K(G,s) \cong K(G,s')$, then equating orders yields $\kappa(s) = \kappa(s')$. If this holds for all vertices $s$ and $s'$, then $M_G := \gcd \{\kappa(v)|v \in V\} = \kappa(s)$ for all $s \in V$.

($2 \implies 1$) If all coordinates of $\kappa$ are equal, then $\Delta \one = \zero$ by Lemma~\ref{lem:kerper}. 
\end{proof}

In particular, the sandpile group $K(G,s)$ is independent of the choice of sink if and only if $G$ is Eulerian (the ``if'' direction is well known \cite[Lemma 4.12]{HLMPPW}). 
% Theorem \ref{thm:groupiso} can be seen as a generalization.

\begin{proof}[Proof of Theorem \ref{t.main}]
($3\implies 4$) We prove the contrapositive. Assume there is a sink
$s$ such that $K(G,s)\neq\left\langle \overline{\beta}_{s}\right\rangle $,
and fix the number of chips on $G$ to be $m=\#E-\#V$. Our assumption implies that there are two distinct cosets $C_{1}$ and $C_{2}$ of $\langle \gamma_s \rangle$ such that
$\tilde{\sigma}_{\max}\in C_{1}$. Choosing an $\eta\in C_{2}$,
we remark that $\eta_{m}$ stabilizes if and only if $\eta_{m}\equiv\sigma_{\max}\mod\Delta\mathbb{Z}^{n}$
(since $\sigma_{\max}$ is the only stable total configuration
with $m$ chips). By Lemma~\ref{cycleequiv}, this congruence holds if and only if $\eta\in\tilde{\sigma}_{\max}\oplus\left\langle \gamma_s \right\rangle =C_{1}$, so we see that $\eta_{m}$ does not stabilize.

($4 \implies 5$) Trivial.

($5\implies 3$) Let $\sigma$ be a total configuration with $\left|\sigma\right|\leq\#E-\#V$. We first write $\sigma=\tau-\delta$ for some $\tau,\delta\in\Z^n$ where $|\tau|=\#E-\#V$ and $\delta\geq 0$. Now if $\left\langle \overline{\beta_s}\right\rangle =K(G,s)$ for some vertex $s$, then $\tilde{\tau} \equiv \tilde{\sigma}_{\max}\mod\Delta_{s}\mathbb{Z}^{n-1}+\mathbb{Z}\beta_s$ so that $\tau\equiv\sigma_{\max}\mod\Delta\mathbb{Z}^{n}$ by Lemma \ref{cycleequiv}. It follows that $\sigma\equiv\sigma_{\max}-\delta\mod\Delta\mathbb{Z}^{n}$. Using that $\Delta \Z^n = \Delta \N^n$, we conclude from Lemma~\ref{l.lap} that $\sigma$ stabilizes. 

%($2\iff 4$ and $2\iff 5$) These equivalences follow from Lemma \ref{lem:order}.

($1\iff 5$) This equivalence follows from Theorem \ref{thm:groupiso}.

($2\iff 5$)  This equivalence follows from Lemma \ref{lem:order}.
\end{proof}

\subsection{Graphs that are both Eulerian and coEulerian}

We conclude this section by characterizing the graphs that are both Eulerian and coEulerian.  
A strongly connected graph $G$ without loops is called a \emph{directed cactus} \cite{Zel89,Sch94} if each edge of $G$ is contained in a unique simple directed cycle. Let us call this the ``unique cycle property'' (UCP). As the proof of the next Proposition will show, the UCP is equivalent to the following ``unique path property'' (UPP): For any pair of vertices $x,y \in V$ there is a unique simple directed path in $G$ from $x$ to $y$. 
(A simple cycle or path is one with no repeated vertices; in particular, a $2$-cycle consisting of an edge and its reversal is simple.)
%We will refer to these two properties respectively as the ``unique cycle property'' (UCP) and ``unique path property'' (UPP). 

Directed cacti are in some sense analogous to trees. In particular, a bidirected graph has the UCP if and only if it is a bidirected tree.

\begin{prop}
\label{p.cactus}
Let $G$ be a strongly connected finite graph without loops. Then $G$ is both Eulerian and coEulerian if and only if $G$ is a directed cactus.
\end{prop}
% note: UPP implies no parallel edges.  Also, loops have no affect on the Laplacian, so adding loops to a directed cactus still leaves it Eulerian and coEulerian.

\begin{proof}
%Eul, coEul implies UPP:

Supposing that $G$ is both Eulerian and coEulerian, we have $\kappa(y)=1$ for all $y \in V$. By Wilson's algorithm \cite{Wil96}, any oriented spanning forest, one of whose components is oriented toward $y$, can be completed to a spanning tree oriented toward $y$. Given $x,y \in V$ and a simple path $P$ from $x$ to $y$, completing $P$ in this manner results in the unique spanning tree $T_y$ oriented toward $y$. Therefore all simple paths from $x$ to $y$ are contained in $T_y$, so $G$ has the UPP.
%Namely, designate a vertex $s$ which has outdegree zero in the forest. For each $v \neq s$ there is a path from $v$ to $s$ in the unique spanning tree $T_s$ oriented toward $s$.  Add to the forest the loop erasure of this path.

% UPP implies UCP

Next observe that for each directed edge $e=(y,x)$ there is a bijection between simple directed paths $P$ from $x$ to $y$ and simple directed cycles $P \cup \{e\}$ containing $e$. Hence the UPP 
%for all pairs $x,y$ such that $(y,x) \in E$ 
implies the UCP.

%UCP implies Eul, coEul:

Finally, supposing that $G$ has the UCP, we will compete the proof by showing that for each vertex $y$ there is a unique spanning tree of $G$ oriented toward $y$, so that $G$ is both Eulerian and coEulerian.
%Then any two simple cycles of $G$ share at most one vertex.  
Let $\mathcal{C}$ be the set of simple directed cycles in $G$, and consider the undirected bipartite graph $\mathcal{T}$ on vertex set $V \cup \mathcal{C}$ whose edges are the pairs $\{v,C\}$ such that vertex $v$ lies on cycle $C$. The UCP implies that $\mathcal{T}$ is a tree. Now we can manifestly describe the unique spanning tree of $G$ oriented toward $y$. Namely, for each cycle $C$ let $e_C$ be the outgoing edge from $x$ in $C$, where $(C,x,\ldots,y)$ is the unique path from $C$ to $y$ in $\mathcal{T}$. 
%omit from each cycle $C$ the outgoing edge from $x$ in $C$, where $(C,x,\ldots,y)$ is the unique path from $C$ to $y$ in $\mathcal{T}$.  
The remaining edges $E - \{e_C|C \in \mathcal{C}\}$ form a spanning tree of $G$ oriented toward $y$.  Moreover any such spanning tree $T$ must contain all edges of $G$ of the form $(x,x')$ where $(C,x,C',\ldots,y)$ is the path from any cycle $C$ to $y$ in $\mathcal{T}$ and $x' \in C'$, else there would be no path from $x$ to $y$ in $T$. By the UCP, edges $e_C$ and $(x,x')$ are distinct since they belong to distinct cycles. Therefore $T$ must omit all of the edges $e_C$, and hence $T$ is unique.
\end{proof}

\section{Computational complexity}
\label{s.NPcomplete}

Bj\"{o}rner and Lov\'{a}sz \cite[Corollary 4.9]{BL92} showed that the \hyperlink{d.halting}{\textsc{halting problem for chip-firing}} can be decided in polynomial time for simple Eulerian graphs. By Theorem \ref{t.main}, it can be decided in linear time for coEulerian multigraphs. The purpose of this section is to show that despite these two easy cases, the problem is $\NP$-complete for general directed multigraphs. 

To see that it is in $\NP$, let $\sigma$ be a nonnegative halting chip configuration on a strongly connected directed multigraph $G=(V,E)$, and let $\xx(v)$ be the number of times vertex $v$ fires.  By Lemma~\ref{l.lap} the vector $\xx$ is a certificate that $\sigma$ halts. Why does this certificate have polynomial size? By Lemma~\ref{l.bl} we have $\xx(v) < \pi(v)$ for some vertex $v$. Moreover for any directed edge $(u_1, u_2)$ the vertex $u_2$ receives at least $\xx(u_1)$ chips from $u_1$ and so $u_2$ fires at least $\xx(u_1) / d_{u_2}$ times. For any vertex $u$, by inducting along a path from $u$ to $v$ we find that $\xx(u) \leq D \xx(v)$ where $D = \prod_{w \in V} d_w$ is the product of all outdegrees. By Lemmas~\ref{lem:period} and \ref{lem:kerper} relating the primitive period vector $\pi$ to the spanning tree count vector $\kappa$, we have $\pi(v) \leq \kappa(v) \leq D$, so all entries of $\xx$ are at most $D^2$. Noting that $\log D \leq \sum_{u,v \in V} \log d_{uv}$, which is the size of description of the adjacency matrix, we conclude that \textsc{the halting problem for chip-firing} is in $\NP$.

To show that it is also $\NP$-hard, our starting point is the following decision problem considered by Amini and Manjunath \cite{AM}. \hypertarget{d.rank}~

\decision{Nonnegative rank}{a basis of an $(n-1)$-dimensional lattice $L \subset \Z^n_0$ and a vector $\sigma\in\Z^n$,}{whether there is a vector $\tau\in\N^{n}$ such
that $\sigma-\tau\in L$.}
If there exists such a $\tau$, then $\sigma$ is said to have \emph{nonnegative rank} relative to $L$. 
In \cite[Theorem 7.2]{AM} \textsc{nonnegative rank} is shown to be $\NP$-hard by reducing from the problem of deciding whether a given simplex with rational vertices contains an integer point.
(To give a little context, the term ``rank'' is inspired by the Riemann-Roch theorem of Baker and Norine \cite{BN}.  Asadi and Backman \cite{AB} extend parts of the Baker-Norine theory to directed graphs. Kiss and T\'{o}thm\'{e}r\'{e}sz \cite{KT} show that computing the Baker-Norine rank---a harder problem than deciding whether it is nonnegative---is already $\NP$-hard when $L$ is the Laplacian lattice of a simple undirected graph.)

The link between chip-firing and \textsc{nonnegative rank} is provided by the following variant of a theorem of Perkinson, Perlman and Wilmes \cite{primer}.

\begin{thm}
\label{t.latticelaplacian}
	Given an $(n-1)$-dimensional lattice $L\subset \Z_{0}^{n}$, there exists a strongly connected multigraph with Laplacian $\Delta$ such that 
		\[ L=\Delta\Z^n. \] 
Moreover, $\Delta$ can be computed from a basis of $L$ in polynomial time, and all entries of $\Delta$ are bounded in absolute value by $nd$ where $d = \det L$.
\end{thm}

The inspiration for Theorem~\ref{t.latticelaplacian} is \cite[Theorem 4.11]{primer}, which expresses an arbitrary $(n-1)$-dimensional lattice in $\Z^{n-1}$ as a \emph{reduced} Laplacian lattice $\Delta_s \Z^{n-1}$.  Modifying its proof to express $L \subset \Z_0^n$ as a \emph{total} Laplacian lattice is straightforward; we give the details below. 

In our application it will be essential to compute the Laplacian matrix $\Delta$ from a basis of $L$ in polynomial time (in the length of description of the basis).   It is not evident whether \cite[Algorithm 4.13]{primer} runs in polynomial time, due to possible blow up of the matrix entries in repeated applications of the Euclidean algorithm \cite{F,HM}.  
As detailed below, this numerical blow up can be avoided by the usual trick of computing modulo the determinant $d$.
 
To see how we will apply Theorem~\ref{t.latticelaplacian}, note that strong connectivity implies 	
	\[ \Delta \Z^n = \Delta \N^n \]
since the period vector of Lemma~\ref{lem:period} is strictly positive. Thus, a vector $\sigma \in \Z^n$ has nonnegative rank relative to $L = \Delta \Z^n$ if and only if there exists $\xx \in \N^n$ such that
	\begin{equation*} \label{e.outofdebt} \sigma + \Delta \xx \geq \zero.  \end{equation*}  
Now by Lemma~\ref{l.lap}, such an $\xx$ exists if and only if the chip configuration $\sigma_{\max} - \sigma$ stabilizes. To summarize, a polynomial time computation of $\Delta$ given a basis for $L$ yields a polynomial time Karp reduction from \textsc{nonnegative rank} to the \textsc{halting problem for chip-firing} on a finite directed multigraph, showing that the latter is $\NP$-hard.

\begin{cor}
\label{c.NPhard}
\textsc{the halting problem for chip-firing} 
%on a finite directed multigraph 
is $\NP$-complete.
\end{cor}

It remains to prove Theorem~\ref{t.latticelaplacian}.  Recall that an $m \times m$ integer matrix $U$ is called \emph{unimodular} if $\det U  = \pm 1$. Any nonsingular square integer matrix $A$ has a \emph{Hermite normal form} 
	\[ H= AU \] 
where $U$ is a unimodular integer matrix, and $H = (h_{ij})$ is lower-triangular with integer entries satisfying
\begin{align*}
&0 < h_{ii}, &&\qquad 1\leq i\leq m \\
&0\leq  h_{ij}<h_{ii}, &&\qquad 1\leq j<i\leq m.
\end{align*}
The existence and uniqueness of $H$ was proved by Hermite \cite{H}. The Hermite normal form is useful to us because $H$ can be computed from $A$ in polynomial time \cite{DKT} and $H\Z^m = A (U \Z^m) = A \Z^m$ by the unimodularity of $U$.  Let \[ d = |\det A | = \det H = \prod_{i=1}^m h_{ii}. \] 
We will use the following observations about the column span $A \Z^m$.
 
\begin{lem}\label{lem:lowtri} \cite[Cor.\ 2.3]{DKT}
	Let $B$ be a lower triangular $m \times m$ matrix whose columns are in
	$A\Z^{m}$ and whose diagonal entries satisfy $b_{ii}=h_{ii}$ for all $i$. Then $B\Z^{m} = A\Z^{m}$. 
\end{lem}

\begin{lem}\label{lem:dei} \cite[Prop.\ 2.5]{DKT} 
	$d \Z^{m} \subset A \Z^{m}$.
\end{lem}

We will apply these lemmas with $m=n-1$. Note that an $n\times n$ integer matrix is the total Laplacian of a directed multigraph if and only if (i) the entries of each column sum to zero, (ii) the diagonal entries are nonnegative, and (iii) the off-diagonal entries are nonpositive.

Given an $n\times (n-1)$ integer matrix whose columns are a $\Z$-basis of the $(n-1)$-dimensional lattice $L \subset \Z_0^n$, let $A$ be the result of removing the last row of $M$.
Since each column of $M$ sums to zero, $A$ is nonsingular. Let $H = AU$ be the Hermite normal form of the $(n-1) \times (n-1)$ nonsingular matrix $A$.

The hypotheses of Lemma \ref{lem:lowtri} are trivially satisfied when $B=H$; and if $B$ satisfies the hypotheses of Lemma \ref{lem:lowtri} then by Lemma \ref{lem:dei} it will continue to do so if we subtract $d$ from an entry below the diagonal. Using this operation we can make the entries immediately below the diagonal sufficiently negative so that the sum of the entries in each column is nonpositive. Namely, let $B = (b_{ij})$ where
	\[ b_{ij} = \begin{cases} h_{ij} - k_j d, & i=j+1 \\ h_{ij}, & \text{else} \end{cases} \]
and $k_j$ for each $j=1,\ldots,n-2$ is a nonnegative integer such that 
\begin{equation}\label{eqn:bound1}
 (k_j - 1)d < \sum_{i=1}^m h_{ij} \leq k_j d .
\end{equation}
Now let 
	\[ \Delta = \left[ \begin{array}{ccccccc} +d & -h_{11} & 0 & 0 & 0 & \cdots & 0 \\
					     0 & + & -h_{22} & 0 & 0 & \cdots & 0 \\
					     0 & - & + & -h_{33} & 0 & \cdots & 0 \\
					     0 & - & - & + & -h_{44} & \cdots & 0 \\
	%				     0 & - & - & - & - & \cdots & 0 \\
					     \vdots& \vdots & \vdots & \vdots & \ddots & \ddots & \vdots \\
					     0 & - & - & - & - & \ddots & -h_{mm} \\
					     -d & - & - & - & - & \cdots & +h_{mm} \end{array} \right] \]
be the $n \times n$ matrix with upper right corner $-B$, the column vector $d\mathbf{e}_{1}-d\mathbf{e}_{n}$ appended on the left, and a row appended on the bottom such that the entries of each column sum to zero. By the choice of $k_j$ in (\ref{eqn:bound1}), the bottom row of $\Delta$ is nonpositive, except for its rightmost entry $h_{mm}$. Therefore $\Delta$ satisfies the conditions (i)-(iii) above. Since the entries immediately above the diagonal of $\Delta$ are negative, as is $\Delta_{n1}$, the matrix $\Delta$ is the Laplacian of a strongly connected multigraph (it has the Hamiltonian cycle $1 \rightarrow n \rightarrow n-1 \rightarrow \cdots \rightarrow 1$). By Lemma~\ref{lem:dei} the first column of $\Delta$ belongs to $L$. Moreover, since both $L$ and $\Delta \Z^n$ are contained in $\Z_0^n$ and $B \Z^{n-1} = A\Z^{n-1}$, the integer span of the remaining columns of $\Delta$ is $L$. Thus $L = \Delta \Z^n$. 
Since each entry of $H$ is at most $d = \prod h_{ii}$, each entry of $\Delta$ has magnitude at most $nd$, 
%by \eqref{eqn:bound1}, 
completing the proof of Theorem~\ref{t.latticelaplacian}.

\subsection{Simple directed graphs}

Let us point out a sense in which the $\NP$-hardness of Corollary~\ref{c.NPhard} is rather weak:  
the directed graphs for which the \hyperlink{d.halting}{\textsc{halting problem for chip-firing}} is hard may have large edge multiplicities.  This is because the Laplacian $\Delta$ of Theorem~\ref{t.latticelaplacian} may have large entries, which in turn is because the lattice $L$ in a hard instance of \hyperlink{d.rank}{\textsc{nonnegative rank}} has large determinant. 
%The adjacency matrix of a directed multigraph on $n$ vertices whose edge multiplicities are bounded by $2^n$ can be specified in only $O(n^3)$ bits, but it is perhaps not so surprising that it is hard to determine whether a chip configuration stabilizes on such a multigraph. 
An interesting question is whether the \textsc{halting problem for chip-firing} remains $\NP$-hard when restricted to \emph{simple} directed graphs, those with edge multiplicities in $\{0,1\}$. Does the hardness arise from directedness or from large edge multiplicities (or both)? The following table summarizes what is known.
	\begin{table}[here]
	\begin{tabular}{c|l|l|}
	~ & \emph{simple graphs} & \emph{multigraphs} \\ \hline
	\emph{coEulerian} & $\P$ (Theorem~\ref{t.main}) & $\P$ (Theorem~\ref{t.main}) \\ \hline
	\emph{bidirected} & $\P$ (Tardos \cite{Tardos}) &  $?$ \\ \hline
	\emph{Eulerian} & $\P$ (Bj\"{o}rner-Lov\'{a}sz \cite{BL92}) & $?$ \\  \hline
	\emph{strongly connected} & $?$ & $\NP$-complete (Cor.~\ref{c.NPhard}) \\
	\hline
	\end{tabular}
	\medskip
	\caption{Complexity of the \protect\hyperlink{d.halting}{\textsc{halting problem for chip-firing}} for eight different classes of strongly connected directed graphs. \label{table.complexity}}
	\end{table}

\section*{Acknowledgements}

The authors thank Spencer Backman, Swee Hong Chan, Daniel Jerison, Robert Kleinberg, Claire Mathieu and John Wilmes for inspiring discussions.  We thank B\'{a}lint Hujter, Viktor Kiss and Lilla T\'{o}thm\'{e}r\'{e}sz for pointing out to us the two open cases in the right column of Table~\ref{table.complexity}. Finally, we thank the referee for comments that meaningfully improved the paper, and in particular for leading us to Propostion~\ref{p.cactus}.

\end{document}

%% file: nodes2.tex
\tikzstyle{vert}=[circle,draw,inner sep=0pt,minimum size=10mm]
\tikzstyle{cdots}=[circle,inner sep=0pt,minimum size=10mm]
\tikzstyle{arrow}=[<-,>=latex,semithick]

\begin{tikzpicture}[node distance=23mm,above,bend right]
	
\node[vert] (node1) {$0$};
\node[vert] (node2) [right of=node1] {$1$};
\node[cdots] (node3)   [right of=node2] {$\cdots$};
\node[vert] (node4)    [right of=node3] {$n$};

\draw [arrow] (node1) [below] to node {2} node {} (node2);
\draw [arrow] (node2) to node {3} node {} (node1);
\draw [arrow] (node2) [below] to node {2} node {} (node3);
\draw [arrow] (node3) to node {3} node {} (node2);
\draw [arrow] (node3) [below] to node {2} node {} (node4);
\draw [arrow] (node4) to node {3} node {} (node3);

\end{tikzpicture}